\def\mE{\mathbb{E}}
\documentclass[11pt]{article}
\usepackage{mathrsfs,amsmath,amsfonts,amssymb,bm,bbm,eufrak,dsfont,pifont,amscd,
stmaryrd,euscript,amsthm,appendix,color,epsfig,xr,yfonts,tikz,verbatim}
\usepackage[round]{natbib}
\usepackage{graphicx,color,authblk}
\usepackage{amsfonts,setspace}
\usepackage{bbm}
\usepackage{amsmath}
\usepackage{tikz}
\newcommand{\bb}{\mathbb}

\newcommand{\Cal}{\mathcal}

\usepackage{hyperref}       
\usepackage{url}            
\usepackage{booktabs}       
\usepackage{amsfonts}       
\usepackage{nicefrac}       
\usepackage{microtype}      %
\newtheorem{theorem}{Theorem}

\newtheorem{proposition}[theorem]{Proposition}
\newtheorem{corollary}[theorem]{Corollary}

\newtheorem{remark}{Remark}
\newcommand{\indep}{\;\, \rule[0em]{.03em}{.67em} \hspace{-.25em}
\rule[0em]{.65em}{.03em} \hspace{-.25em}\rule[0em]{.03em}{.67em}\;\,}

\newcommand{\X}{\mathscr{X}}
\newcommand{\dX}{\ddot{X}}

\newcommand{\Y}{\mathscr{Y}}
\newcommand{\V}{\mathcal{V}}

\renewcommand{\H}{\mathscr{H}}

\def\0{{\bf 0}}
\def\1{{\bf 1}}

\def\D{{\mathcal{D}}}

\def\M{\mbox{$\mathcal{M}$}}

\def\Z{{\mathscr{Z}}}

\def\bq{\begin{equation}}
\def\eq{\end{equation}}

\def\squarebox#1{\hbox to #1{\hfill\vbox to #1{\vfill}}}

\def\bse{\begin{eqnarray*}}
\def\ese{\end{eqnarray*}}
\def\be{\begin{eqnarray}}
\def\ee{\end{eqnarray}}
\def\bsq{\begin{equation*}}
\def\esq{\end{equation*}}
\def\bq{\begin{equation}}
\def\eq{\end{equation}}

\def\boxit#1{\vbox{\hrule\hbox{\vrule\kern6pt\vbox{\kern6pt#1\kern6pt}\kern6pt\vrule}\hrule}}

\providecommand{\keywords}[1]
{
  {\small	
  \textbf{\textit{Keywords---}} #1
}}
\title{On Distance and Kernel Measures of Conditional Independence}
\author{Tianhong Sheng\thanks{txs514@psu.edu}  }
\author{Bharath K. Sriperumbudur\thanks{bks18@psu.edu}}
\affil{Department of Statistics, Pennsylvania State University\\
University Park, PA 16802, USA.}
\date{}
\begin{document}

\maketitle

\begin{abstract}
Measuring conditional independence is one of the important tasks in statistical inference and is fundamental in causal discovery, feature selection, dimensionality reduction, Bayesian network learning, and others. In this work, we explore the connection between conditional independence measures induced by distances on a metric space and reproducing kernels associated with a reproducing kernel Hilbert space (RKHS). For certain \emph{distance and kernel pairs}, we show the distance-based conditional independence measures to be equivalent to that of kernel-based measures. On the other hand, we also show that some popular---in machine learning---kernel conditional independence measures based on the Hilbert-Schmidt norm of a certain cross-conditional covariance operator, do not have a simple distance representation, except in some limiting cases. This paper, therefore, shows the distance and kernel measures of conditional independence to be not \emph{quite equivalent} unlike in the case of joint independence as shown by \citet{sejdinovic2013equivalence}.
\end{abstract}
\keywords{energy distance, distance covariance, reproducing kernel Hilbert space, Hilbert-Schmidt independence criterion, conditional independence test}

\section{Introduction}
Measuring conditional independence between random variables plays a fundamental role in many statistical inference tasks such as causal discovery \citep{pearl2000causality, spirtes2000causation}, 
supervised dimensionality reduction \citep{cook2002dimension, fukumizu2004dimensionality}, conditional independence testing \citep{su2007consistent, gretton2012kernel}, and others. 
Formally, for random variables $X,Y$, and $Z$, $X$ is said to be \emph{conditionally independent} of $Y$ given $Z$, denoted as $X\indep Y|Z$, if $\mathbb{P}_{XY|Z}=\mathbb{P}_{X|Z}\mathbb{P}_{Y|Z}$. Given a distance measure $D$ on the space of probability measures, $D(\mathbb{P}_{XY|Z},\mathbb{P}_{X|Z}\mathbb{P}_{Y|Z})$ measures the degree of conditional (in)dependence between $X$ and $Y$ given $Z$, with $X\indep Y|Z$ if and only if $D(\mathbb{P}_{XY|Z},\mathbb{P}_{X|Z}\mathbb{P}_{Y|Z})=0$. Some popular choices for $D$ include the Kullback-Leibler divergence (more generally $f$-divergence), total variation distance, Hellinger distance, Wasserstein distance, among others. 

Recently, a class of distances on probability measures induced by a Euclidean metric on $\mathbb{R}^d$---more generally by metrics of strongly negative type---, called the energy distance \citep{Szekely2004} and distance covariance \citep{Szekely-07,SzeRiz09,lyons2013distance} has gained popularity in nonparametric hypothesis testing (e.g., two-sample and independence testing), because of their computational simplicity and elegant interpretation. \citet{wang2015conditional} extended distance covariance to conditional distributions on $\bb{R}^d$ to obtain a measure of conditional independence, called \emph{conditional distance covariance} (CdCov) and has been applied in conditional independence testing. We refer to these class of probability metrics as \emph{distance-based measures} and point the reader to Section~\ref{Sec:dcov} for preliminaries on distance-based measures.

On the other hand, in the machine learning literature, measures of independence have been formulated based on embedding of probability distributions into a reproducing kernel Hilbert space (RKHS; \citealp{aronszajn1950theory}). This embedding into RKHS allows to capture the properties of distributions and has been used in many applications including homogeneity, independence, and conditional independence testing (for example, see \citealp{MAL-060} and references therein). Formally, given a probability measure $\nu$ defined on a measurable space $\X$, and a RKHS $\H_{k}$ with the reproducing kernel $k$, $\nu$ can be embedded into $\H_{k}$ as\be
 \nu \mapsto \int_\X k(\cdot,x)\, d\nu(x) := \mu_k(\nu) \in \H_{k},\nonumber
\ee
where $\mu_k(\nu)$ is called the mean element or kernel mean embedding of $\nu$. Using this notion, the \emph{kernel distance}~\citep{gretton2006twosample} between two probability distributions $\bb{P}$ and $\bb{Q}$ is defined as the distance between their mean elements, i.e., $D(\bb{P},\bb{Q})=\Vert \mu_k(\bb{P})-\mu_k(\bb{Q})\Vert_{\H_k}$. The kernel embedding and the kernel distance are well-studied in the literature and their mathematical theory is well-developed \citep{sriperumbudur2010hilbert,sriperumbudur-11-jmlr,sriperumbudur-bernoulli}. Generalizing this notion of kernel embedding to distributions defined on product spaces yields a kernel measure of independence, called the Hilbert-Schmidt independence criterion (HSIC; \citealp{gretton2005measuring}, \citealp{gretton2008kernel}, \citealp{smola2007hilbert}), which can then be used as a measure of conditional independence by employing it to conditional probability distributions \citep{fukumizu2004dimensionality,fukumizu2008kernel}. \citet{fukumizu2004dimensionality,gretton2005measuring} provided an alternate interpretation for HSIC in terms of the Hilbert-Schmidt norm of a certain cross-covariance operator, based on which the Hilbert-Schmidt norm of a conditional cross-covariance operator (we refer to it as HS\"{C}IC) is then proposed as a measure of conditional independence. We point the reader to Sections~\ref{Sec:kernel} and \ref{subsec:relation} for details and refer to these class of probability metrics as \emph{kernel-based measures}.

\cite{sejdinovic2013equivalence} established an equivalence between distance-based and kernel-based independence measures (i.e., distance covariance and HSIC) by showing that a reproducing kernel that defines HSIC induces a semimetric of negative type which in turn defines the distance covariance, and vice-versa. However, despite the striking similarity, the relationship between conditional distance covariance and related kernel measures is not known. The goal of this work is to investigate the relationship between distance and kernel-based measures of conditional independence, and in particular, understand whether these measures are equivalent (i.e., the distance measure can be obtained from the kernel measure and vice-versa).

As our contributions, first, in Theorem~\ref{thm5_1} (Section~\ref{subsec:hscic}), we generalize the conditional distance covariance of \citet{wang2015conditional} to arbitrary metric spaces of negative type---we call this as generalized CdCov (gCdCov)---and develop a kernel measure of conditional independence (we refer to it as HSCIC) that is \emph{equivalent} to gCdCov. Therefore, it follows from Theorem~\ref{thm5_1} that CdCov introduced by \citet{wang2015conditional} is a special case of the HSCIC. Second, in Theorem~\ref{pro:hsc} (Section~\ref{subsec:relation}), we consider the kernel measure of conditional independence based on the Hilbert-Schmidt norm of the conditional cross-covariance operator (i.e., HS\"{C}IC) and obtain its distance-based interpretation. We show that this distance-based version of HS\"{C}IC does not have an elegant interpretation, except in limiting cases where it is related to CdCov and gCdCov 
(see 
Corollaries~\ref{pro:tv} and \ref{cor:2}). Third, in Section~\ref{sec:estimation}, we investigate the similarities and differences between distance-based and kernel-based measures from the point of view of empirical estimation---the empirical estimates of these measures are used in the literature as test statistics in conditional independence testing. We show that unlike distance-based measures that require conditional density estimation, HS\"{C}IC only requires the estimation of linear functionals of joint distributions and therefore does not require estimation of conditional densities, while both have similar computational complexity requirements. Overall, these results establish that unlike in the case of joint independence, the distance and kernel-based measures of conditional independence are not quite equivalent.

The paper is organized as follows. Definitions and notation that are widely used throughout the paper are collected in Section~\ref{Sec:notation}. The preliminaries on distance-based and kernel-based measures are presented in Sections~\ref{Sec:dcov} and \ref{subsec:prelim}, respectively, while main results are presented in Sections~\ref{subsec:hscic}, \ref{subsec:relation}, and \ref{sec:estimation}.


\section{Definition and Notation}\label{Sec:notation}
For a non-empty set $\X$, a function $\rho: \X \times \X \to [0, \infty)$ is called a semimetric on $\X$ if it satisfies  (i) $\rho(x,x')= 0 \Leftrightarrow x=x'$ and (ii) $\rho(x,x')=\rho(x',x)$. Then $(\X, \rho)$ is said to be a semimetric space. The semimetric space, $(\X, \rho)$ is said to be of negative type if $\forall n \geq 2$, $\{x_i\}^n_{i=1} \in \X$, and $\{\alpha_i\}^n_{i=1}\in \bb{R}$, with $\sum_{i=1}^n \alpha_i = 0$,  $\sum_{i=1}^n\sum_{j=1}^n \alpha_i \alpha_j \rho(z_i,z_j) \leq 0$. $(\X,\rho)$ is said to be of strongly negative type if for all finite signed measures $\mu$ such that $\mu(\X)=0$, $\int\int \rho(x,y)\,d\mu(x)\,d\mu(y)<0$ for all $\mu\ne 0$. 

For a non-empty set $\X$, a real-valued symmetric function $k:\X \times \X \to \bb{R}$ is called a positive definite (pd) kernel if, for all $n \in \bb{N}$, $\{\alpha_i\}^n_{i=1} \in \bb{R}$ and $\{x_i\}^n_{i=1} \in \X$, we have $\sum_{i,j=1}^n \alpha_i \alpha_j k(x_i,x_j) \geq 0$. A function $k:\X \times \X \to \bb{R}$, $(x,y) \mapsto k(x,y)$ is a \textit{reproducing kernel} of the Hilbert space $(\H_k, \langle\cdot,\cdot \rangle_{\H_k})$ of functions if and only if (i) $\forall x \in \X$, $k(\cdot, x) \in \H$ and (ii) $\forall x \in \X$, $\forall f \in \H_k$, $\langle k(\cdot, x), f\rangle_{\H_k} = f(x)$ hold. If such a $k$ exists, then $\H_k$ is called a \textit{reproducing kernel Hilbert space}. It is easy to show that every reproducing kernel (r.k.) is symmetric and positive definite, since $\langle k(\cdot, x), k(\cdot, y)\rangle_{\H_k} = k(x,y), \forall x,y \in \X$.

In this paper, $\X$, $\Y$ and $\Z$ denote measurable spaces endowed with Borel $\sigma$-algebras. $X$, $Y$ and $Z$ denote random elements in $\X$, $\Y$ and $\Z$, respectively. $\dX$ is defined as $(X,Z)$, which is a random element in $\X \times \Z$. The probability law of a random variable $X$ is denoted by $P_X$ and the joint probability law of random variables $X$ and $Z$ is denoted by $P_{XZ}$. A measurable, positive definite kernel on $\X$ is denoted as $k_{\X}$ and its corresponding RKHS as $\H_{\X}$. Similarly we define  $k_{\Y}$, $\H_{\Y}$, $k_{\Z}$, $\H_{\Z}$, $k_{\dX}$ and $\H_{\dX}$. 
$\M(\Z)$ denotes the set of all finite signed Borel measures on $\Z$, $\M^{1}_{+}(\Z)$ denotes the set of all finite probability measures on $\Z$, $\M^{\theta}_{k}(\Z):=\{v \in \M(\Z)| \int k^{\theta}(z,z)d|v|(z) < \infty\}$ and $\M^{\theta}_{\rho}(\Z):=\{v \in \M(\Z)|\int \rho^{\theta}(z,z_0)~d|v|(z) < \infty \text{ for some } z_0 \in \Z\}$. The space of $r$-integrable functions w.r.t.~a $\sigma$-finite measure, $\mu$ on $\mathbb{R}^d$ is denoted as $L^r(\mathbb{R}^d,\mu)$ and if $\mu$ is a Lebesgue measure on $\mathbb{R}^d$, we denote it as $L^r(\mathbb{R}^d)$.

The symbol $X \indep Y|Z $ indicates the conditional independence of $X$ and $Y$ given $Z$. $\phi_{X}$ and $\phi_{Y}$ denote the characteristic functions of $X$ and $Y$ respectively and their joint characteristic function is denoted as $\phi_{XY}$. The conditional characteristic functions of $X$, $Y$ and $(X,Y)$ given $Z$ are denoted as $\phi_{X|Z}$, $\phi_{Y|Z}$ and $\phi_{XY|Z}$ respectively. 
\section{Conditional Distance Covariance}\label{Sec:dcov}
Distance covariance was proposed by \cite{szekely2007measuring} as a new measure of dependence between random vectors in arbitrary dimension. An interesting feature of distance covariance 
is that unlike the classical 
covariance, it is zero only if the random vectors are independent. Formally, the distance covariance (dCov) between two random vectors is defined as the weighted $L^2$ norm between the joint characteristic function and the product of marginal characteristic functions, i.e., $$\V^2(X,Y) = \| \phi_{XY}- \phi_{X}\phi_{Y} \|^2_{L^2(w)}= \frac{1}{c_p c_q} \int\int \frac{|\phi_{XY}(t,s) - \phi_{X}(t)\phi_{Y}(s)|^2}{\Vert t\Vert^{p+1} \Vert s\Vert^{q+1}}\,dt\,ds,$$ where $\phi_{XY}$ denotes the joint characteristic function of random variables $X\in \mathbb{R}^p$ and $Y\in\mathbb{R}^q$ with $\phi_X$ and $\phi_Y$ denoting their respective marginal characteristic functions. Here $c_p=\frac{\pi^{(p+1)/2}}{\Gamma((p+1)/2)}$, $c_q=\frac{\pi^{(q+1)/2}}{\Gamma((q+1)/2)}$ and $w(t,s)=\Vert t\Vert^{-p-1}\Vert s\Vert^{-q-1}$ with $\Vert t\Vert^2=\sum^p_{i=1}{t^2_i}$ for $t=(t_1,\ldots,t_p)$. A particular advantage of distance covariance is its compact representation in terms of certain expectation of pairwise Euclidean distances (\citealp{szekely2007measuring}):
\begin{eqnarray}\label{Eq:dcov}
\V^2(X,Y)&{}={}&
\mE_{XY}\mE_{X'Y'}\|X-X'\| \|Y-Y'\| + \mE_{X}\mE_{X'} \|X-X'\|\mE_{Y}\mE_{Y'} \|Y-Y'\| \nonumber\\
&{}{}&\qquad\quad-2 \mE_{XY}[\mE_{X'} \|X-X'\|\mE_{Y'} \|Y-Y'\|],
\end{eqnarray}
 which leads to straightforward empirical estimates by replacing the expectations with empirical estimators. Such an estimator has been used as a test statistic in independence testing and the resulting test is shown to be consistent if the marginal distributions have finite first moment (\citealp{szekely2007measuring}). 
 As a natural generalization, \cite{lyons2013distance} extended \eqref{Eq:dcov} to metric spaces of negative type and showed that the corresponding distance covariance---obtained by replacing the Euclidean metric by a metric of strongly negative type---is zero if and only if $X$ and $Y$ are independent. 

Extending the idea of distance covariance, recently, \cite{wang2015conditional} proposed a conditional version to measure conditional independence between random vectors of arbitrary dimension. To elaborate, 
let $X\in\mathbb{R}^p$, $Y\in\mathbb{R}^q$ and $Z\in\mathbb{R}^r$ be random vectors. 
The conditional distance covariance (CdCov) $\V(X,Y|Z)$ between random vectors $X$ and $Y$ with finite moments given $Z$ is defined as 
$$
\V^2(X,Y|Z) = \| \phi_{XY|Z} - \phi_{X|Z}\phi_{Y|Z} \|^2_{L^2(w)}
=  \int\int \frac{|\phi_{XY|Z}(t,s) - \phi_{X|Z}(t)\phi_{Y|Z}(s)|^2}{c_p c_q\Vert t\Vert^{p+1} \Vert s\Vert^{q+1}}\,dt\,ds,
$$
where 
$\phi_{XY|Z}(t,s) = \mE\left[e^{\sqrt{-1}\langle t,X\rangle + \sqrt{-1} \langle s,Y\rangle}|Z\right]$, $\phi_{X|Z}(t) = \phi_{XY|Z}(t,0)$ and $\phi_{Y|Z}(s) = \phi_{XY|Z}(0,s)$.
As a crucial property, CdCov is zero $P_Z$-almost surely if and only if  $X \indep Y|Z $. Similar to distance covariance, one advantage of this measure is that its sample version can be expressed elegantly as a $V$- or $U$-statistic, based on which
\cite{wang2015conditional} proposed a statistically consistent conditional independence test.

The conditional distance covariance defined above can also be computed in terms of the conditional expectations of pairwise Euclidean distances:
\begin{eqnarray}
\V^2(X,Y|Z)&{}={}&
\mE_{XY|Z}\mE_{X'Y'|Z}\|X-X'\| \|Y-Y'\| + \mE_{X|Z}\mE_{X'|Z} \|X-X'\|\mE_{Y|Z}\mE_{Y'|Z} \|Y-Y'\| \nonumber\\
&{}{}&\qquad- 2 \mE_{XY|Z}[\mE_{X'|Z} \|X-X'\|\mE_{Y'|Z} \|Y-Y'\|],\label{Eq:cond-dist}
\end{eqnarray}
where $X,Y|Z$ and $X',Y'|Z$ $\overset{\text{i.i.d.}}{\sim} P_{XY|Z}$. As carried out in \cite{lyons2013distance}, CdCov can be extended to metric spaces of negative type so that \eqref{Eq:cond-dist}
can be written as 
\begin{eqnarray}\label{Eq:metric-cond}
\V^2_{\rho_{\X},\rho_{\Y}}(X,Y|Z)&{} ={}&
\mE_{XY|Z}\mE_{X'Y'|Z}\rho_{\X}(X,X')\rho_{\Y}(Y,Y')
+ \mE_{X|Z}\mE_{X'|Z} \rho_{\X}(X,X')\mE_{Y|Z}\mE_{Y'|Z} \rho_{\Y}(Y,Y') \nonumber\\
&{}{}&\qquad\quad -2 \mE_{XY|Z}[\mE_{X'|Z} \rho_{\X}(X,X')\mE_{Y'|Z} \rho_{\Y}(Y,Y')],
\end{eqnarray}
where $\rho_\X$ and $\rho_\Y$ are metrics of strongly negative type defined on spaces $\X$ and $\Y$ respectively with $X|Z \sim P_{X|Z} \in \M^{2}_{\rho_{\X}}(\X)$ and $Y |Z\sim P_{Y|Z} \in \M^{2}_{\rho_{\Y}}(\Y)$ for any $Z$. 
The moment conditions ensure that the expectations are finite. When $\rho_\X$ and $\rho_\Y$ are strongly negative, then clearly \eqref{Eq:metric-cond} is zero if and if $X \indep Y|Z $. Equivalently, gCdCov can be represented in an integral form,
\begin{equation}
\V^2_{\rho_{\X},\rho_{\Y}}(X,Y|Z)= \int (\rho_{\X}\rho_{\Y})(x,y,x',y')\, d[P_{XY|Z}- P_{X|Z} P_{Y|Z}]^2(x,y,x',y'),\nonumber
\end{equation}
where $\rho_{\X}\rho_{\Y}$ is viewed as a function on $(\X \times \Y) \times (\X \times \Y)$ for any fixed $Z$ and $P^2:=P\times P$.

\section{Kernel Measures of Conditional Independence}\label{Sec:kernel}
First, in Section~\ref{subsec:prelim}, we present preliminaries on RKHS embedding of probability measures and introduce kernel measures of independence. Based on this discussion, in Section~\ref{subsec:hscic}, we develop a kernel measure of conditional independence (we call it as Hilbert-Schmidt conditional independence criterion---HSCIC) that is related to gCdCov (and therefore CdCov) discussed in Section~\ref{Sec:dcov}. We also present an interpretation for gCdCov through conditional cross-covariance operator formulation for HSCIC.
\subsection{RKHS embedding of probabilities}\label{subsec:prelim}
In the machine learning literature, the notion of embedding probability measures in an RKHS has gained lot of attention and has been applied in goodness-of-fit \citep{balasubramanian2017optimality}, 
two-sample \citep{gretton2006twosample,gretton2012kernel}, independence \citep{gretton2008kernel} and conditional independence  \citep{fukumizu2008kernel,zhang2011kernel} testing. To elaborate, given a probability measure $P\in\M^{1/2}_k(\X)$, its RKHS embedding is defined \citep{smola2007hilbert} as $$P\mapsto \mu_P:=\int_{\X} k(\cdot,x)\,dP(x)\in \H_k,$$ where $\H_k$ is an RKHS with $k$ as the reproducing kernel. Based on this embedding, a distance on the space of probabilities can be defined through the distance between the embeddings, i.e., $\D_k(P,Q)=\Vert \mu_P-\mu_Q\Vert_{\H_k}$, called the \emph{kernel distance} or \emph{maximum mean discrepancy} \citep{gretton2006twosample}.
If the map $P\mapsto \mu_P$ is injective, then the kernel $k$ that induces $\mu_P$ is said to be \emph{characteristic} \citep{fukumizu2009characteristic,sriperumbudur2010hilbert} and therefore $\D_k(P,Q)$ induces a metric on $\M^{1/2}_k(\X)$. Using the reproducing property of the kernel, it can be shown that
\begin{eqnarray}\label{Eq:mmd}
\D^2_k(P,Q)&{}={}&
\mE_{XX'}k(X,X') + \mE_{YY'}k(Y,Y')-2\mE_{XY}k(X,Y),\nonumber
\end{eqnarray}
where $X,X'\stackrel{i.i.d.}{\sim}P$ and $Y,Y'\stackrel{i.i.d.}{\sim}Q$.
Extending this distance to probability measures on product spaces, particularly the joint measure $P_{XY}$ and product of marginals $P_XP_Y$, yields a measure of dependence between two random variables $X$ and $Y$ defined on measurable spaces $\X$ and $\Y$, called the Hilbert-Schmidt Independence Criterion (HSIC), which is defined \citep{gretton2005measuring} as
\begin{eqnarray}\label{Eq:hsic}
\D^2_{k_\X k_\Y}(P_{XY},P_XP_Y)&{}={}&
\mE_{XY}\mE_{X'Y'}k_\X(X,X')k_\Y(Y,Y') + \mE_{X}\mE_{X'}k_\X(X,X')\mE_{Y}\mE_{Y'}k_\Y(Y,Y')\nonumber\\
&{}{}&\qquad-2\mE_{XY}[\mE_{X'}k_\X(X,X') \mE_{Y'}k_\Y(Y,Y')]\\
&{}={}&\int (k_{\X}k_\Y)(x,y,x',y')\,d[P_{XY}- P_{X} P_{Y}]^2(x,y,x',y'),\nonumber
\end{eqnarray}
where the product kernel $k_\X k_\Y$ is a reproducing kernel for the tensor RKHS $\H_\X\otimes\H_\Y$ and $(X',Y')$ is an independent copy of $(X, Y)$. If the kernels $k_\X$ and $k_\Y$ are characteristic, then HSIC characterizes independence \citep{szabo2018charactersitc}, i.e., $\D_{k_\X k_\Y}(P_{XY},P_XP_Y)$ $=0$ if and only if $X\indep Y$. 
An empirical version of \eqref{Eq:hsic} has been used as a test statistic in independence testing and the resultant test is shown to be consistent against all alternatives as long as $k_\X$ and $k_\Y$ are characteristic \citep{gretton2008kernel}. An interesting connection between kernel-based HSIC and distance-based dCov is shown by \citet{sejdinovic2013equivalence} that dCov in \eqref{Eq:dcov} is in fact a special case of HSIC and HSIC is equivalent to the generalized dCov introduced by \citet{lyons2013distance}.
This result provides a unifying framework for the distance and kernel-based independence measures. With this background, in the rest of the paper, we explore the relation between distance and kernel-based measures of conditional independence.

\subsection{Hilbert-Schmidt conditional independence criterion}\label{subsec:hscic}
For appropriate choice of kernels and distances, the following result provides a kernel-equivalent of gCdCov, which we refer to as the Hilbert-Schmidt conditional independence criterion (HSCIC).
\begin{theorem}\label{thm5_1}
Let $(\X,\rho_{\X})$ and $(\Y,\rho_{\Y})$ be semimetric spaces of negative type. Suppose $X|Z=z \sim P_{X|Z=z} \in \M^2_{\rho_{\X}}(\X)$ and $Y|Z=z \sim P_{Y|Z=z} \in \M^2_{\rho_{\Y}}(\Y)$, having joint distribution $P_{XY|Z=z}$ for all $z\in \Z$. If $k_{\X}$ and $k_{\Y}$ are pd kernels on $\X$ and $\Y$ that are distance-induced, i.e.,
$$k_\X(x,x')=\rho_\X(x,\theta)+\rho_\X(x',\theta)-\rho_\X(x,x')$$
and 
$$k_\Y(y,y')=\rho_\Y(y,\theta')+\rho_\Y(y',\theta')-\rho_\Y(y,y')$$
for some $\theta\in\X$ and $\theta'\in\Y$. Then
\begin{eqnarray}
\V^2_{\rho_{\X},\rho_{\Y}}(X,Y|Z=z) &{}={}& 
\D^2_{k_\X k_\Y}(P_{XY|Z=z}, P_{X|Z=z} P_{Y|Z=z}).\label{Eq:equiv0}
\end{eqnarray}
On the other hand, suppose $k_\X$ and $k_\Y$ are pd kernels on $\X$ and $\Y$ respectively with $X|Z=z\sim P_{X|Z=z}\in \M^2_{k_{\X}}(\X)$ and $Y|Z=z \sim P_{Y|Z=z} \in \M^2_{k_{\Y}}(\Y)$, having joint distribution $P_{XY|Z=z}$ for all $z\in \Z$. If $\rho_\X$ and $\rho_\Y$ are semi-metrics on $\X$ and $\Y$ that are kernel-induced, i.e.,
$$\rho_\X(x,x')=\frac{k_\X (x,x)+k_\X (x',x')}{2}-k_\X (x,x')$$ and $$\rho_\Y(y,y')=\frac{k_\Y (y,y)+k_\Y (y',y')}{2}-k_\Y (y,y'),$$
then \eqref{Eq:equiv0} holds.
\end{theorem}
\begin{proof}
Define $v_z:=P_{XY|Z=z}-P_{X|Z=z} P_{Y|Z=z}$. Suppose $k_\X$ and $k_\Y$ are distance-induced. Then
\begin{eqnarray}
\D^2_{k_\X k_\Y}(P_{XY|Z=z},P_{X|Z=z} P_{Y|Z=z})&{}={}& \int \int k_{\X}(x,x')k_{\Y}(y,y')\,dv_z(x,y)\,dv_z(x',y'),\nonumber\\
&{}={}&\int\int \left(\rho_\X(x,\theta)+\rho_\X(x',\theta)-\rho_\X(x,x')\right)\nonumber\\
&{}{}&\qquad\times \left(\rho_\Y(y,\theta')+\rho_\Y(y',\theta')-\rho_\Y(y,y')\right)\,dv_z(x,y)\,dv_z(x',y'),\nonumber\\
&{}={}&\int \int
\rho_{\X}(x,x') \rho_{\Y}(y,y')\,dv_z(x,y)\,dv_z(x',y')\nonumber\\
&{}={}&\V^2_{\rho_{\X},\rho_{\Y}}(X,Y|Z=z),\nonumber
\end{eqnarray}
where we used the facts that $v_z(\X \times \Y)=0$ for any $z$ and $\int g(x,y,x',y')\,dv_z(x,y)\,dv_z(x',y')=0$ when $g$ does not depend on one or more of its arguments, since $v_z$ also has zero marginal measures. On the other hand, suppose $\rho_\X$ and $\rho_\Y$ are kernel-induced. Clearly they are of negative type. Then
\begin{eqnarray}
\V^2_{\rho_{\X},\rho_{\Y}}(X,Y|Z=z) &{}={}&\int \int
\rho_{\X}(x,x') \rho_{\Y}(y,y')\,dv_z(x,y)\,dv_z(x',y')\nonumber\\
&{}={}& \int \int \left(\frac{k_{\X}(x,x) + k_{\X}(x',x')}{2} - k_{\X}(x,x')\right) \nonumber\\
&{}{}&\qquad\times \left(\frac{k_{\Y}(y,y) + k_{\Y}(y',y')}{2} -k_{\Y}(y,y')\right) \,dv_z(x,y)\,dv_z(x',y')\nonumber\\
&{}={}& \int \int k_{\X}(x,x')k_{\Y}(y,y')\,dv_z(x,y)\,dv_z(x',y'),\nonumber\\
&{}={}&\D^2_{k_\X k_\Y}(P_{XY|Z=z},P_{X|Z=z} P_{Y|Z=z}),\nonumber
\end{eqnarray}
where we again used the above mentioned facts about $v_z$.
\end{proof}
Since $\rho_\X$ and $\rho_\Y$ are induced by the Euclidean norm, choosing $k_\X(x,x')=\Vert x\Vert+\Vert x'\Vert-\Vert x-x'\Vert,\,x,x'\in\bb{R}^p$ and $k_\Y(y,y')=\Vert y\Vert+\Vert y'\Vert-\Vert y-y'\Vert,\,y,y'\in\bb{R}^q$ yields CdCov.

While HSCIC is a natural measure of conditional independence, in the kernel literature, however, a different measure has been widely used \citep{fukumizu2004dimensionality,fukumizu2008kernel,zhang2011kernel}, which is based on the Hilbert-Schmidt norm of a certain \emph{conditional cross-covariance operator}. Before we introduce the conditional cross-covariance operator and these other measures of conditional independence (which we do in Section~\ref{subsec:relation}) , first we will briefly discuss how HSIC is related to the Hilbert-Schmidt norm of a cross-covariance operator so that its extension to the conditional version is natural. 

For random variables $X\sim P_X$ and $Y\sim P_Y$ with joint distribution $P_{XY}$ such that $P_X\in\M^1_{k_\X}(\X)$ and $P_Y\in\M^1_{k_\Y}(\Y)$, there exists a unique bounded linear operator, called the cross-covariance operator \citep{baker1973joint,fukumizu2004dimensionality}, 
$\Sigma_{YX}:\H_{k_\X}\rightarrow\H_{k_\Y}$ such that $$\langle g, \Sigma_{YX} f \rangle_{\H_{k_{\Y}}}= E_{XY}[f(X)g(Y)]-E_{X}[f(X)]E_{Y}[g(Y)],\,\,\forall\,f\in\H_{k_\X},\,g\in\H_{k_\Y}.$$ In fact, using the reproducing property that $f(x)=\langle f,k_\X(\cdot,x)\rangle_{\H_{k_\X}},\,\forall\,x\in\X$ and $g(y)=\langle g,k_\Y(\cdot,y)\rangle_{\H_{k_\Y}},\,\forall\,y\in\Y$, it follows that \begin{equation}\Sigma_{YX}=\int\int k_\Y(\cdot,y)\otimes k_\X(\cdot,x)\,dP_{XY}(x,y)-\int k_\Y(\cdot,y)\,dP_Y(y) \otimes \int k_\X(\cdot,x)\,dP_X(x),\label{Eq:cov-op}\end{equation}
where $\otimes$ denotes the tensor product. Clearly, $\Sigma_{YX}$ is a natural generalization of the finite-dimensional covariance matrix between two random vectors $X\in \mathbb{R}^p$ and $Y\in \mathbb{R}^q$. Based on \eqref{Eq:cov-op} and the reproducing property, it can be verified that 
\begin{eqnarray}
\Vert \Sigma_{YX}\Vert^2_{HS}&{}={}& \left\Vert \int\int k_\X(\cdot,x)\otimes k_\Y(\cdot,y)\,d(P_{XY}-P_XP_Y)(x,y)\right\Vert^2_{HS}\nonumber\\
&{}={}&\int\int\int\int\langle k_\X(\cdot,x)\otimes k_\Y(\cdot,y),k_\X(\cdot,x')\otimes k_\Y(\cdot,y')\rangle_{HS}\,d(P_{XY}-P_XP_Y)(x,y)\nonumber\\
&{}{}&\qquad\qquad\qquad\times\,d(P_{XY}-P_XP_Y)(x',y')\nonumber\\
&{}={}&\D^2_{k_\X k_\Y}(P_{XY},P_X P_Y),\label{Eq:simi}
\end{eqnarray}
where $\Vert\cdot\Vert_{HS}$ denotes the Hilbert-Schmidt norm. Since HSCIC is a conditional version of HSIC and since the latter is the Hilbert-Schmidt norm of the cross-variance operator, it is natural to extend $\Sigma_{YX}$ to its conditional version as the bounded linear operator $\dot{\Sigma}_{YX|Z}:\H_{k_\X}\rightarrow\H_{k_\Y}$ (actually this is not just one operator, but a collection of operators indexed by $Z$) such that
$$\langle g, \dot{\Sigma}_{YX|Z} f \rangle_{\H_{k_{\Y}}}= E_{XY|Z}[f(X)g(Y)|Z]-E_{X|Z}[f(X)|Z]E_{Y|Z}[g(Y)|Z],\,\,\forall\,f\in\H_{k_\X},\,g\in\H_{k_\Y},$$
thereby yielding
\begin{equation}\dot{\Sigma}_{YX|Z}=\int\int k_\Y(\cdot,y)\otimes k_\X(\cdot,x)\,dP_{XY|Z}(x,y)-\int k_\Y(\cdot,y)\,dP_{Y|Z}(y) \otimes \int k_\X(\cdot,x)\,dP_{X|Z}(x).\nonumber
\end{equation}
Similar to \eqref{Eq:simi}, it is easy to verify that $$\Vert \dot{\Sigma}_{YX|Z}\Vert^2_{HS}=\D^2_{k_{\X} k_{\Y}}(P_{XY|Z},P_{X|Z}P_{Y|Z}).$$ Therefore if $k_\X$ and $k_\Y$ are characteristic, then $X\indep Y|Z \Longleftrightarrow \dot{\Sigma}_{YX|Z}=0$, $P_Z$-a.s.

However, in the kernel literature, to the best of our knowledge, HSCIC has not been used as a measure of conditional independence probably because it is not a single measure but a family of measures indexed by $Z$---one can obtain a single measure of conditional independence by considering the weighted average of HSCIC, i.e., \begin{equation}\Cal{D}_\Lambda(P_{XY|Z},P_{X|Z}P_{Y|Z}):=\int \Vert \dot{\Sigma}_{YX|Z=z}\Vert^2_{HS}\,d\Lambda(z),\label{Eq:single}
\end{equation}
where $\Lambda$ is a finite positive measure on $\Z$, which can be chosen to be $P_Z$. Instead, a related version of HSCIC has been widely used, which we will discuss and explore its connection to gCdCov in the following section.
\section{Relation between RKHS and distance-based conditional independence measures}\label{subsec:relation}
Since $\dot{\Sigma}_{YX|Z}$ is a family of operators (indexed by $Z$), \citet{fukumizu2004dimensionality} considered an alternate operator, called the \emph{conditional cross-covariance operator}, which is defined as follows. Suppose $P_X\in \M^1_{k_\X}(\X)$, $P_Y\in \M^1_{k_\Y}(\Y)$ and $P_Z\in \M^1_{k_\Z}(\Z)$. Then there exists a unique bounded linear operator $\Sigma_{YX|Z}$ such that
\begin{eqnarray}\label{Eq:cond-cov}
\langle g, \Sigma_{YX|Z} f \rangle_{\H_{k_{\Y}}}&=& \mathbb{E}_{XY}[f(X)g(Y)]-\mathbb{E}_Z[\mathbb{E}_{X|Z}[f(X)]\mathbb{E}_{Y|Z}[g(Y)]]\nonumber\\
&=&\mathbb{E}_Z[\text{Cov}_{XY|Z}(f(X),g(Y)|Z)]\nonumber
\end{eqnarray}
for all $f\in\H_{k_\X}$ and $g\in\H_{k_\Y}$. As above, using the reproducing property, it can be shown that
$$\Sigma_{YX|Z}=\int\int\int k_\Y(\cdot,y)\otimes k_\X(\cdot,x)\,d\left[P_{XY|Z=z}-P_{X|Z=z}P_{Y|Z=z}\right](x,y)\,dP_Z(z)=\mathbb{E}_Z[\dot{\Sigma}_{YX|Z}].$$
However, unlike $\dot{\Sigma}_{YX|Z}$, the conditional cross-covariance operator $\Sigma_{YX|Z}$ does not characterize conditional independence since $\Sigma_{YX|Z}=0$---assuming $k_\X$ and $k_\Y$ to be characteristic---only implies $P_{XY}=\mathbb{E}_Z[P_{X|Z}P_{Y|Z}]$ and not $\dot{\Sigma}_{YX|Z}=0$, $P_Z$-a.s. (\citealp[Theorem 8]{fukumizu2004dimensionality}). Therefore, \citet[Corollary 9]{fukumizu2004dimensionality} considered $Z$ as a part of $X$ by defining $\ddot{X}:=(X,Z)$ and showed that $\Sigma_{Y\ddot{X}|Z}=0$ if and only if $X\indep Y|Z$, assuming $k_\X$, $k_\Y$ and $k_\Z$ to be characteristic. This is indeed the case since if $k_\X$, $k_\Y$ and $k_\Z$ are characteristic, then $\Sigma_{Y\ddot{X}|Z}=0$ implies $\mathbb{E}_Z[\dot{\Sigma}_{Y\ddot{X}|Z}]=0$ and therefore $\mathbb{E}_Z[P_{Y\ddot{X}|Z}]=\mathbb{E}_Z[P_{Y|Z}P_{\ddot{X}|Z}]$, i.e., 
\begin{eqnarray}
\mathbb{E}_Z[\mathbb{E}_{XY|Z}[\mathds{1}\{X\in A, Y\in B\}|Z\in C ]\mathds{1}\{Z\in C\}]&{}={}&\mathbb{E}_Z[\mathbb{E}_{X|Z}[\mathds{1}\{X\in A\}|Z\in C]\nonumber\\
&&\qquad\times \mathbb{E}_{Y|Z}[\mathds{1}\{Y\in B\}|Z\in C]\mathds{1}\{Z\in C\}]\nonumber
\end{eqnarray}
for all $A\in \mathcal{B}_\X$, $B\in\mathcal{B}_\Y$ and $C\in\mathcal{B}_\Z$, where $\mathcal{B}_\X$, $\mathcal{B}_\Y$ and $\mathcal{B}_\Z$ are the Borel $\sigma$-algebras associated with $\X$, $\Y$ and $\Z$ respectively. This implies $P_{XY|Z}(A\times B)=(P_{X|Z}P_{Y|Z})(A\times B)$ for all $A\in\mathcal{B}_\X$ and $B\in\mathcal{B}_\Y$, $P_Z$-a.s., implying that $X\indep Y|Z$. Hence $\Vert \Sigma_{Y\ddot{X}|Z}\Vert^2_{HS}$ can be used as a measure of conditional independence, which we refer to it as HS\"{C}IC. 

The goal of this section is to explore the distance counterpart of HS\"{C}IC and understand how it is related to CdCov, gCdCov, and $\mathcal{D}_\Lambda$ defined in \eqref{Eq:single}. To this end, we first provide an expression for $\Vert \Sigma_{Y\ddot{X}|Z}\Vert^2_{HS}$ in terms of kernels, using which we obtain an expression in terms of distances.
\begin{theorem} \label{pro:hsc}
Let $P_X\in \M^2_{k_\X}(\X)$, $P_Y\in \M^2_{k_\Y}(\Y)$ and $P_Z\in \M^1_{k_\Z}(\Z)$ with joint distribution $P_{XYZ}$. 
Denote $\ddot{X}=(X,Z)$ and $v_z=P_{XY|Z=z}-P_{X|Z=z} P_{Y|Z=z}$. Then
\begin{eqnarray}
\|\Sigma_{Y\ddot{X}|Z}\|^2_{HS}&{}={}& 
\int\int k_\Z(z,z')\left\langle \dot{\Sigma}_{YX|Z=z},\dot{\Sigma}_{YX|Z=z'}\right\rangle_{HS}\,dP_Z(z)\,dP_Z(z')\nonumber\\
&{}={}&{\displaystyle \int\int }k_{\Z}(z,z') h(z,z')\, d P_{Z}(z)\,d P_{Z}(z'),\label{Eq:maineq}
\end{eqnarray}
where
$h(z,z')= {\displaystyle \int \int} k_{\X}(x,x')k_{\Y}(y,y')\, d v_z(x,y)\, d v_{z'}(x',y').$
Suppose $k_\X$ and $k_\Y$ are distance-induced, i.e.,
$$k_\X(x,x')=\rho_\X(x,\theta)+\rho_\X(x',\theta)-\rho_\X(x,x')\,\,\,\text{and}\,\,\,k_\Y(y,y')=\rho_\Y(y,\theta')+\rho_\Y(y',\theta')-\rho_\Y(y,y')$$
for some $\theta\in\X$ and $\theta'\in\Y$.
Then
$$h(z,z')=\int\int \rho_\X(x,x') \rho_\Y(y,y')\,dv_z(x,y)\,dv_{z'}(x',y').$$
\end{theorem}
\begin{proof}
Note that 
\begin{eqnarray}
\Sigma_{Y\ddot{X}|Z}&{}={}&\mathbb{E}_Z[\dot{\Sigma}_{Y\ddot{X}|Z}]\nonumber\\
&{}={}&\mathbb{E}_Z[\mathbb{E}_{Y\ddot{X}|Z}[k_\Y(\cdot,Y)\otimes (k_\X k_\Z)(\cdot,\ddot{X})|Z]]\nonumber\\
&{}{}&\qquad-\mathbb{E}_Z[\mathbb{E}_{Y|Z}[k_\Y(\cdot,Y)|Z]\otimes \mathbb{E}_{\ddot{X}|Z}[(k_\X k_\Z)(\cdot,\ddot{X})|Z]]\nonumber\\
&{}={}& \mathbb{E}_Z[\mathbb{E}_{YX|Z}[k_\Y(\cdot,Y)\otimes k_\X (\cdot,X)|Z]\otimes k_\Z(\cdot,Z)]\nonumber\\
&{}{}&\qquad-\mathbb{E}_Z[\mathbb{E}_{Y|Z}[k_\Y(\cdot,Y)|Z]\otimes \mathbb{E}_{X|Z}[k_\X (\cdot,X)|Z]\otimes k_\Z(\cdot,Z)]\nonumber\\
&{}={}&\mathbb{E}_Z[\dot{\Sigma}_{YX|Z}\otimes k_\Z(\cdot,Z)].\nonumber
\end{eqnarray}
Therefore,
\begin{eqnarray}
\Vert \Sigma_{Y\ddot{X}|Z}\Vert^2_{HS}&{}={}&\left\Vert \mathbb{E}_Z[\dot{\Sigma}_{YX|Z}\otimes k_\Z(\cdot,Z)]\right\Vert^2_{HS}=\left\langle \mathbb{E}_Z[\dot{\Sigma}_{YX|Z}\otimes k_\Z(\cdot,Z)], \mathbb{E}_Z[\dot{\Sigma}_{YX|Z}\otimes k_\Z(\cdot,Z)]\right\rangle_{HS}\nonumber\\
&{}={}&\mathbb{E}_Z\mathbb{E}_{Z'}\left\langle \dot{\Sigma}_{YX|Z}\otimes k_\Z(\cdot,Z),\dot{\Sigma}_{YX|Z'}\otimes k_\Z(\cdot,Z')\right\rangle_{HS}\nonumber\\
&{}={}&\mathbb{E}_Z\mathbb{E}_{Z'}\left\langle
\dot{\Sigma}_{YX|Z},\dot{\Sigma}_{YX|Z'}\right\rangle_{HS}\langle k_\Z(\cdot,Z),k_\Z(\cdot,Z')\rangle_{\H_{k_{\Z}}}\nonumber\\
&{}={}&\mathbb{E}_Z\mathbb{E}_{Z'}\left\langle
\dot{\Sigma}_{YX|Z},\dot{\Sigma}_{YX|Z'}\right\rangle_{HS}k_\Z(Z,Z').\label{Eq:hs-norm}
\end{eqnarray}
Note that $\dot{\Sigma}_{YX|Z}=\int k_\Y(\cdot,y)\otimes k_\X(\cdot,x)\,dv_z(x,y)$ and therefore 
\begin{eqnarray}
\left\langle
\dot{\Sigma}_{YX|Z},\dot{\Sigma}_{YX|Z'}\right\rangle_{HS}&{}={}&\left\langle\int k_\Y(\cdot,y)\otimes k_\X(\cdot,x)\,dv_z(x,y),\int k_\Y(\cdot,y)\otimes k_\X(\cdot,x)\,dv_{z'}(x,y) \right\rangle_{HS}\nonumber\\
&{}={}&\int\int \left\langle k_\Y(\cdot,y)\otimes k_\X(\cdot,x),k_\Y(\cdot,y')\otimes k_\X(\cdot,x') \right\rangle\,dv_z(x,y)\,dv_{z'}(x',y')\nonumber\\
&{}={}&\int\int \left\langle k_\Y(\cdot,y), k_\Y(\cdot,y')\right\rangle_{\H_{k_\Y}}\left\langle k_\X(\cdot,x), k_\X(\cdot,x') \right\rangle_{\H_{k_\X}}\,dv_z(x,y)\,dv_{z'}(x',y')\nonumber\\
&{}={}&\int\int k_\X(x.x')k_\Y(y,y')\,dv_z(x,y)\,dv_{z'}(x',y')=:h(z,z'),\nonumber
\end{eqnarray}
using which in \eqref{Eq:hs-norm} yields the result. If $k_\X$ and $k_\Y$ are distance-induced, then using the fact that $\int g(x,x',y,y')\,dv_z(x,y)\,dv_{z'}(x',y')=0$ when $g$ does not depend on one or more of its arguments---basically, the same argument that we carried out in the proof of Theorem~\ref{thm5_1}---we have 
$$h(z,z')=\int\int \rho_\X(x,x') \rho_\Y(y,y')\,dv_z(x,y)\,dv_{z'}(x',y'),$$
and the result follows.
\end{proof}
While $h(z,z')$ has a distance interpretation as shown in Theorem~\ref{pro:hsc}, $\Vert \Sigma_{Y\ddot{X}|Z}\Vert^2_{HS}$ does not have an elegant representation in terms of distances. Suppose $k_\Z$ is also distance-induced, i.e.,  $k_\Z(z,z')=\rho_\Z(z,\theta^{''})+\rho_\Z(x',\theta^{''})-\rho_\Z(z,z')$
for some $\theta^{''}\in\Z$. Then
\begin{eqnarray}
\Vert \Sigma_{Y\ddot{X}|Z}\Vert^2_{HS}&{}={}&\int\int h(z,z')k_\Z(z,z')\,dP_Z(z)\,dP_Z(z')\nonumber\\
&{}={}&\int\int \left[\rho_\Z(z,\theta^{''})+\rho_\Z(x',\theta^{''})-\rho_\Z(z,z')\right]h(z,z')\,dP_Z(z)\,dP_Z(z').\label{Eq:temp}
\end{eqnarray}
Unfortunately, \eqref{Eq:temp} cannot be related in a simple manner to gCdCov or HSCIC. However, some simplifications occur based on certain assumptions on $k_\Z$, as shown in the following corollaries. Under an appropriate choice of $k_\Z$, Corollary~\ref{pro:tv} shows HS\"{C}IC to be asymptotically equivalent to the weighted average of HSCIC (equivalently, the weighted average of gCdCov) defined in \eqref{Eq:single} while Corollary~\ref{cor:2} shows the asymptotic equivalence between HS\"{C}IC and CdCov. 
\begin{corollary}\label{pro:tv}
Suppose the assumptions of Theorem~\ref{pro:hsc} hold and $P_Z$ has a density $p_Z$ w.r.t.~the Lebesgue measure on $\mathbb{R}^d$ such that $h(z,\cdot)p_Z$ is uniformly continuous and bounded for all $z\in\mathbb{R}^d$. For $t>0$, let $$k_\Z(z,z')=\frac{1}{t^d}\psi\left(\frac{z-z'}{t}\right),\,z,z'\in\mathbb{R}^d,$$ where $\psi\in L^1(\mathbb{R}^d)$ is a bounded continuous positive definite function with $\int_{\mathbb{R}^d}\psi(z)\,dz=1$. Then
\begin{eqnarray}\lim_{t\rightarrow 0}\Vert \Sigma_{Y\ddot{X}|Z}\Vert^2_{HS}&{}={}&
\int \Vert \dot{\Sigma}_{YX|Z=z}\Vert^2_{HS}\,p^2_Z(z)\,dz=\Cal{D}_{\Lambda}(P_{XY|Z},P_{X|Z}P_{Y|Z})\nonumber
\end{eqnarray}
with $d\Lambda=p^2_Z\,dz$, where $\Cal{D}_\Lambda$ is defined in \eqref{Eq:single}.
\end{corollary}
\begin{proof}
Define $\psi_t(z):=t^{-d}\psi\left(\frac{z}{t}\right)$.
From \eqref{Eq:maineq}, it follows that
\begin{eqnarray}
\Vert \Sigma_{Y\ddot{X}|Z}\Vert^2_{HS}&{}={}&\int\int \psi_t(z-z')h(z,z')p_Z(z)p_Z(z')\,dz\,dz'\nonumber\\
&{}={}& \int p_Z(z)\left(\int \psi_t(z-z')h(z,z')p_Z(z')\,dz'\right)\,dz=\int p_Z(z)(\psi_t*(h(z,\cdot)p_Z)(z)\,dz,\nonumber
\end{eqnarray}
where $*$ denotes convolution. Taking the limit on both sides as $t\rightarrow 0$ and applying dominated convergence theorem, we obtain
\begin{eqnarray}
\lim_{t\rightarrow 0}\Vert \Sigma_{Y\ddot{X}|Z}\Vert^2_{HS}&{}={}&\lim_{t\rightarrow 0}\int p_Z(z)(\psi_t*(h(z,\cdot)p_Z)(z)\,dz
=\int p_Z(z)\lim_{t\rightarrow 0}(\psi_t*(h(z,\cdot)p_Z)(z)\,dz.\nonumber\end{eqnarray}
The result follows from \citet[Theorem 8.14]{Folland-99} which yields $\lim_{t\rightarrow 0}(\psi_t*(h(z,\cdot)p_Z)(z)=h(z,z)p_Z(z)$ for all $z\in\mathbb{R}^d$ and by noting that $h(Z,Z)=\Vert \dot{\Sigma}_{YX|Z}\Vert^2_{HS}$.
\end{proof}
\begin{corollary}\label{cor:2}
Suppose the assumptions of Theorem~\ref{pro:hsc} hold with $\rho_\X(x,x')=\Vert x-x'\Vert,\,x,x'\in\mathbb{R}^p$ and $\rho_\Y(y,y')=\Vert y-y'\Vert,\,y,y'\in\mathbb{R}^q$. Let $k(z,z')=\eta(z)\eta(z'),\,z,z'\in\mathbb{R}^d$ for some real-valued function $\eta$ on $\mathbb{R}^d$ and
\begin{equation}
\int \left|\eta(z)\right| \left\Vert\phi_{XY|Z=z} - \phi_{X|Z=z} \phi_{Y|Z=z}\right\Vert_{L^2(w)}\,dP_Z(z)<\infty.\label{Eq:cond-dc}    
\end{equation}
Then
\begin{eqnarray}\Vert \Sigma_{Y\ddot{X}|Z}\Vert^2_{HS}=\left\Vert \int \eta(z) \left(\phi_{XY|Z=z} - \phi_{X|Z=z} \phi_{Y|Z=z}\right)\,dP_Z(z)\right\Vert^2_{L^2(w)},\label{Eq:cor4-1}
\end{eqnarray}
where $w(t,s)=\frac{1}{c_pc_q}\Vert t\Vert^{-p-1}\Vert s\Vert^{-q-1},\,t\in\mathbb{R}^p,\,s\in\mathbb{R}^q$. In particular, for $t>0$ and some $a\in\mathbb{R}^d$, if $\eta(z)=\frac{1}{t^d}\theta\left(\frac{a-z}{t}\right),\,z\in\bb{R}^d$ where $\theta$ is a bounded continuous function with $\int \theta(z)\,dz=1$ and $P_Z$ has a bounded uniformly continuous density $p_Z$~on $\bb{R}^d$ such that \begin{equation}\int \sup_{z}\left|\phi_{XY|Z=z}(t,s) - \phi_{X|Z=z}(t) \phi_{Y|Z=z}(s)\right|^2\,dw(t,s)<\infty,\label{Eq:cond-dc-1}\end{equation} then \begin{equation}\lim_{t\rightarrow 0}\Vert \Sigma_{Y\ddot{X}|Z}\Vert^2_{HS}= p^2_Z(a)\V^2(X,Y|Z=a).\label{Eq:final1}\end{equation}
\end{corollary}
\begin{proof}
In the following, we show that 
\begin{equation}h(z,z')=\left\langle\phi_{XY|Z=z} - \phi_{X|Z=z} \phi_{Y|Z=z}, \phi_{XY|Z=z'} - \phi_{X|Z=z'} \phi_{Y|Z=z'} \right\rangle_{L^2(w)}\label{Eq:tempo}\end{equation} and therefore \eqref{Eq:cor4-1} follows by using \eqref{Eq:tempo} in \eqref{Eq:maineq} with $k(z,z')=\eta(z)\eta(z')$ and applying dominated convergence theorem through \eqref{Eq:cond-dc}. 
We now prove \eqref{Eq:tempo}. Consider
\begin{eqnarray}
&&\left\langle \phi_{XY|Z=z} - \phi_{X|Z=z} \phi_{Y|Z=z}, \phi_{XY|Z=z'} - \phi_{X|Z=z'} \phi_{Y|Z=z'} \right\rangle_{L^2(w)}\nonumber\\
&{}={}& \int\int w(t,s)\left[\phi_{XY|Z=z}(t,s) - \phi_{X|Z=z}(t) \phi_{Y|Z=z}(s)\right]\left[\overline{\phi_{XY|Z=z'}(t,s) - \phi_{X|Z=z'}(t) \phi_{Y|Z=z'}(s)}\right]\,dt\,ds\nonumber\\
&{}={}&  \int \int w(t,s) \Lambda(t,s)\, dt\,ds,\label{Eq:ip}
\end{eqnarray}
where
\begin{eqnarray}
\Lambda(t,s) &{}={}& \left[\phi_{XY|Z=z}(t,s) - \phi_{X|Z=z}(t) \phi_{Y|Z=z}(s)][\overline{\phi_{XY|Z=z'}(t,s) - \phi_{X|Z=z'}(t) \phi_{Y|Z=z'}(s)}\right]\nonumber\\
&{}={}& \left[E_{XY|Z=z}e^{i(\langle t, x \rangle + \langle s, y \rangle)}- E_{X|Z=z}e^{i\langle t,x \rangle}E_{Y|Z=z}e^{i\langle s,y\rangle }\right]\nonumber\\
&{}{}&\qquad\qquad \cdot \left[\overline{E_{XY|Z=z}e^{i(\langle t, x \rangle + \langle s, y \rangle)}- E_{X|Z=z}e^{i\langle t,x \rangle}E_{Y|Z=z}e^{i\langle s,y\rangle}}\right]\nonumber\\
&{}={}& E_{XY|Z=z}E_{XY|Z=z'} e^{i(\langle t, x-x' \rangle + \langle s, y-y' \rangle)} - E_{XY|Z=z}E_{X|Z=z'}E_{Y|Z=z'}e^{i(\langle t, x-x' \rangle + \langle s, y-y' \rangle) }\nonumber\\
&{}{}&\qquad\qquad- E_{X|Z=z}E_{Y|Z=z}E_{XY|Z=z'}e^{i(\langle t, x-x' \rangle + \langle s, y-y' \rangle) } \nonumber\\
&{}{}&\qquad\qquad\qquad\qquad+ E_{X|Z=z}E_{Y|Z=z}E_{X|Z=z'}E_{Y|Z=z'}e^{i(\langle t, x-x' \rangle + \langle s, y-y' \rangle) }\nonumber\\
&{}={}&\int\int e^{i(\langle t, x-x' \rangle + \langle s, y-y' \rangle) }\,dv_z(x,y)\,dv_{z'}(x',y').\label{Eq:lambda}
\end{eqnarray}
Using \eqref{Eq:lambda} in \eqref{Eq:ip}, we obtain
\begin{equation}\int\int w(t,s)\Lambda(t,s)\,dt\,ds=\int\int\int\int \cos\langle t, x-x' \rangle \cos\langle s, y-y' \rangle w(t,s)\,dv_z(x,y)\,dv_{z'}(x',y')\,dt\,ds \label{Eq:tt}\end{equation}
by noting that $\sin\langle t, x-x' \rangle$ and $\sin\langle s, y-y' \rangle$ are odd functions w.r.t.~$t$ and $s$ respectively.
Since $\cos\langle t, x-x' \rangle \cos\langle s, y-y' \rangle  =  1-(1-\cos \langle t, x-x' \rangle) - (1- \cos \langle s, y-y' \rangle ) + (1-\cos \langle t, x-x' \rangle)(1- \cos \langle s, y-y' \rangle )$ and $$\int f(x,x',y,y')\, dv_z(x,y)\,dv_{z'}(x',y')=0$$ for $f(x,x',y,y')=1$, $f(x,x',y,y')=1-\cos\langle t,x-x'\rangle$ and $f(x,x',y,y')=1-\cos\langle s,y-y'\rangle$, \eqref{Eq:tt} reduces to
\begin{eqnarray}\int\int w(t,s)\Lambda(t,s)\,dt\,ds&{}={}&\int\int\int\int \frac{1-\cos \langle t, x-x' \rangle}{c_p\Vert t\Vert^{p+1}}\cdot\frac{1- \cos \langle s, y-y' \rangle}{c_q\Vert s\Vert^{q+1}}\,dv_z(x,y)\,dv_{z'}(x',y')\,dt\,ds\nonumber\\
&{}={}&\int\int \Vert x-x'\Vert \Vert y-y'\Vert\,dv_z(x,y)\,dv_{z'}(x',y')=h(z,z'),\nonumber
\end{eqnarray}
where the last equality follows from Lemma 1 of \cite{szekely2007measuring} through $\int \frac{1-\cos \langle t, x\rangle }{c_p\Vert t\Vert ^{p+1}}\,dt =\Vert x\Vert$, thereby proving the result in \eqref{Eq:cor4-1}. By defining $\theta_t(z)=t^{-d}\theta\left(\frac{z}{t}\right)$, we have
$$\int \eta(z) \left(\phi_{XY|Z=z} - \phi_{X|Z=z} \phi_{Y|Z=z}\right)\,dP_Z(z)=\theta_t*\left(\left(\phi_{XY|Z} - \phi_{X|Z} \phi_{Y|Z}\right)p_Z\right)(a),$$
which by \citep[Theorem 8.14]{Folland-99} converges to $\left(\phi_{XY|Z=a} - \phi_{X|Z=a} \phi_{Y|Z=a}\right)p_Z(a)$ as $t\rightarrow 0$. Using these in \eqref{Eq:cor4-1} along with dominated convergence theorem combined with \eqref{Eq:cond-dc-1} yields \eqref{Eq:final1}.
\end{proof}
\begin{remark} Informally, the result of Corollary~\ref{pro:tv} can be obtained by choosing $k_\Z(z,z')=\delta(z-z'),\,z,z'\in\mathbb{R}^d$, where $\delta(\cdot)$ is the Dirac distribution. Since such a choice does not correspond to a valid reproducing kernel---Dirac distribution is not a function but a \emph{distribution} that does not belong to an RKHS---, the rigorous argument involves considering a family of kernels indexed by bandwidth $t$ which in the limiting case of $t\rightarrow 0$ achieves the behavior of the Dirac distribution. Similar argument applies to Corollary~\ref{cor:2} as well.

\end{remark}
The discussion so far shows that HS\"{C}IC is a kernel measure of conditional independence that does not possess a clean expression for its distance counterpart except in very special scenarios. However, in the machine learning literature, HS\"{C}IC is more widely used as a measure of conditional independence than $\mathbb{E}_Z\Vert \dot{\Sigma}_{YX|Z}\Vert^2_{HS}=\mathcal{D}_{p_Z}(P_{XY|Z},P_{X|Z}P_{Y|Z})]$ in conditional independence tests. This is because, from the point of view of estimating these measures, HS\"{C}IC enjoys a simple and computationally efficient estimator, as detailed in the following section.
\section{Estimation}\label{sec:estimation}
In this section, we investigate the question of
estimating $\mathcal{D}_{p_Z}(P_{XY|Z},P_{X|Z}P_{Y|Z})=\mathbb{E}_Z[\Vert\dot{\Sigma}_{YX|Z}\Vert^2_{HS}]$ and $\Vert\Sigma_{Y\ddot{X}|Z}\Vert^2_{HS}$, which is of importance as these estimators are used as test statistics in testing for conditional independence. 
In the following, we first provide the basic idea to construct these estimators. To this end, consider the problem of estimating $$\gamma_z:=\int\int a(x,y)\,dP_{XY|Z=z}(x,y)$$ based on samples $(X_i,Y_i,Z_i)^n_{i=1}\stackrel{i.i.d.}{\sim}P_{XYZ}$. Suppose $P_Z$ has a density $p_Z$ w.r.t.~a dominating measure $\mu$ on $(\rho_\Z,\Z)$ where $\rho_\Z$ is the metric on $\Z$. Define $\theta_i(z):=K(\rho_\Z(z,Z_i))$ and $\theta(z):=\sum^n_{i=1}\theta_i(z)$ where $K:\mathbb{R}^+\rightarrow \mathbb{R}^+$ with $\int K(x)\,dx=1$. Then it is clear that conditioned on $(Z_i)^n_{i=1}=z$, $$\frac{1}{\theta(z)}\sum^n_{i=1}a(X_i,Y_i)\theta_i(z)$$ is an unbiased estimator of $\gamma_z$ for all $z\in\Z$. Using this idea, the following estimators for $\mathbb{E}_Z[\Vert\dot{\Sigma}_{YX|Z}\Vert^2_{HS}]$ and $\Vert\Sigma_{Y\ddot{X}|Z}\Vert^2_{HS}$ can be obtained. To this end, consider
\begin{eqnarray}
&{}{}& \mathcal{D}_{p_Z}(P_{XY|Z},P_{X|Z}P_{Y|Z})]=\mathbb{E}_Z[\Vert\dot{\Sigma}_{YX|Z}\Vert^2_{HS}]\nonumber\\
&{}={}& \int \left\Vert \int k_\Y (\cdot,y)\otimes k_\X(\cdot,x)\,d\left(P_{XY|Z=z}-P_{X|Z=z}P_{Y|Z=z}\right)(x,y)\right\Vert^2_{HS}\,dP_Z(z)\nonumber\\
&{}\approx {}& \int \left\Vert \sum^n_{i=1}k_\Y(\cdot,Y_i)\otimes k_\X(\cdot,X_i)\frac{\theta_i(z)}{\theta(z)}-\left(\sum^n_{i=1}k_\Y(\cdot,Y_i)\frac{\theta_i(z)}{\theta(z)}\right)\otimes \left(\sum^n_{i=1}k_\X(\cdot,X_i)\frac{\theta_i(z)}{\theta(z)}\right)\right\Vert^2_{HS}\,dP_Z(z)\nonumber\\
&{}\approx&{}\frac{1}{n}\sum^n_{j=1}\left\Vert \sum^n_{i=1}k_\Y(\cdot,Y_i)\otimes k_\X(\cdot,X_i)\frac{\theta_i(Z_j)}{\theta(Z_j)}-\left(\sum^n_{i=1}k_\Y(\cdot,Y_i)\frac{\theta_i(Z_j)}{\theta(Z_j)}\right)\otimes \left(\sum^n_{l=1}k_\X(\cdot,X_l)\frac{\theta_l(Z_j)}{\theta(Z_j)}\right)\right\Vert^2_{HS}\nonumber\\
&{}={}&\frac{1}{n}\sum^n_{j=1}\left[\sum^n_{a,b=1}k_\X(X_a,X_b)k_\Y(Y_a,Y_b)\frac{\theta_a(Z_j)\theta_b(Z_j)}{\theta^2(Z_j)}\right.\nonumber\\
&{}{}&\qquad\left.-2\sum^n_{a,b,c=1}k_\X(X_a,X_c) k_\Y(Y_a,Y_b)\frac{\theta_a(Z_j)\theta_b(Z_j)\theta_c(Z_j)}{\theta^3(Z_j)}\right]\nonumber\\
&{}{}&\qquad\qquad+\frac{1}{n}\sum^n_{j=1}\sum^n_{a,b,c,d=1}k_\X(X_a,X_c) k_\Y(Y_b,Y_d)\frac{\theta_a(Z_j)\theta_b(Z_j)\theta_c(Z_j)\theta_d(Z_j)}{\theta^4(Z_j)},\label{Eq:estim-ours}
\end{eqnarray}
where the r.h.s.~of \eqref{Eq:estim-ours} is an estimator ($V$-statistics version) of $\mathbb{E}_Z[\Vert\dot{\Sigma}_{YX|Z}\Vert^2_{HS}]$. The symbol $\approx$ indicates that a given line is an approximation to the previous line and is obtained using the aforementioned idea to estimate $\gamma_z$. A $U$-statistic version of \eqref{Eq:estim-ours} can be similarly derived. It has to be noted that both these versions have a computational complexity of $O(n^3)$ and the consistency of these estimators can be established using the standard convergence results in $U$ and $V$-statistics \citep{serfling2009approximation}. \citet{wang2015conditional} obtained an estimator similar to \eqref{Eq:estim-ours} for CdCov and showed it to be consistent (see Sections 3.2 and 3.3 of \citealp{wang2015conditional}).

Similarly, an estimator of $\Vert\Sigma_{Y\ddot{X}|Z}\Vert^2_{HS}$ can be obtained as:
\begin{eqnarray}
&{}{}&\Vert\Sigma_{Y\ddot{X}|Z}\Vert^2_{HS}\nonumber\\
&{}={}&\left\Vert \int k_\Y(\cdot,y)\otimes k_\X(\cdot,x)\otimes k_\Z(\cdot,z)\,d\left(P_{XY|Z=z}-P_{X|Z=z}P_{Y|Z=z}\right)(x,y)\,dP_Z(z)\right\Vert^2_{HS}\nonumber\\
&{}\approx{}&\left\Vert \frac{1}{n}\sum^n_{j=1}\left(\int k_\Y(\cdot,y)\otimes k_\X(\cdot,x)\, d\left(P_{XY|Z=Z_j}-P_{X|Z=Z_j}P_{Y|Z=Z_j}\right)(x,y)\right)\otimes k_\Z(\cdot,Z_j)\right\Vert^2_{HS}\nonumber
\end{eqnarray}
\begin{eqnarray}
&{}\approx{}& \left\Vert\frac{1}{n}\sum^n_{j=1}\left(\sum^n_{i=1}k_\Y(\cdot,Y_i)\otimes k_\X(\cdot,X_i)\frac{\theta_i(Z_j)}{\theta(Z_j)}-\left(\sum^n_{i=1}k_\Y(\cdot,Y_i)\frac{\theta_i(Z_j)}{\theta(Z_j)}\right)\otimes\left(\sum^n_{l=1}k_\X(\cdot,X_l)\frac{\theta_l(Z_j)}{\theta(Z_j)}\right)\right)\right.\nonumber\\
&{}{}&\qquad\qquad\left.\otimes k_\Z(\cdot,Z_j) \right\Vert^2_{HS}\nonumber\\
&{}={}&\frac{1}{n^2}\sum^n_{a,b,c,d=1}k_\X(X_a,X_b)k_\Y(X_a,X_b)k_\Z(Z_a,Z_b)\frac{\theta_a(Z_b)\theta_b(Z_a)}{\theta(Z_a)\theta(Z_b)}\nonumber\\
&{}{}&\qquad -\frac{2}{n^2}\sum^n_{a,b,c,d,e=1}k_\X(X_b,X_d)k_\Y(Y_b,Y_e)k_\Z(Z_a,Z_c)\frac{\theta_b(Z_a)\theta_d(Z_c)\theta_e(Z_c)}{\theta(Z_a)\theta^2(Z_c)}\nonumber\\
&{}{}&\qquad\qquad +\frac{1}{n^2}\sum^n_{a,b,c,d,e,f=1}k_\X(X_b,X_e)k_\Y(Y_c,Y_f)k_\Z(Z_a,Z_d)\frac{\theta_b(Z_a)\theta_c(Z_a)\theta_e(Z_d)\theta_f(Z_d)}{\theta^2(Z_a)\theta^2(Z_d)},\label{Eq:v-hscic}
\end{eqnarray}
which is a $V$-statistic version with a computational complexity of $O(n^4)$, that is higher than that of the estimator in \eqref{Eq:estim-ours} for $\mathbb{E}_Z[\Vert\dot{\Sigma}_{YX|Z}\Vert^2_{HS}]$. However, under certain conditions (see Proposition~\ref{pro:equivalence}), \cite{fukumizu2004dimensionality}
obtained an alternate expression for $\Sigma_{Y\ddot{X}|Z}$, using which a computationally efficient estimator with a complexity of $O(n^3)$ can be obtained for HS\"{C}IC that does not require the estimation of $p_Z$ (see Proposition~\ref{pro:estim}).
\begin{proposition}(\citealp[Proposition 5]{fukumizu2004dimensionality})\label{pro:equivalence}
Suppose
$\mE_X[k^2_{\X}(X,X)]$, $\mE_Y[k_{\Y}(Y,Y)]$ and $\mE_Z[k^2_{\Z}(Z,Z)]$ are finite. If $h\!\cdot\!\mE_{X|Z}[f(X)|Z= \cdot]$ and $\mE_{Y|Z}[g(Y)|Z= \cdot]$ are elements of $\H_{\Z}$ for all $f \in \H_{\X}$, $g \in \H_{\Y}$ and $h\in\H_{\Z}$, then 
\begin{equation}
\Sigma_{Y\ddot{X}|Z} = \Sigma_{Y\ddot{X}} - \Sigma_{YZ} \tilde{\Sigma}_{ZZ}^{-1} \Sigma_{Z\ddot{X}},\label{Eq:equiv}
\end{equation}
where $\tilde{\Sigma}^{-1}_{ZZ}$ is the right inverse of $\Sigma_{ZZ}$ on $(\emph{ker}(\Sigma_{ZZ}))^\perp$.
\end{proposition}
The key observation in Proposition~\ref{pro:equivalence}
is that the conditional covariance operator can be expressed in terms of only covariance operators (this is reminiscent of the situation when $(X,Y,Z)$ are jointly normal so that the conditional covariance matrix can be represented in terms of the joint covariance matrices). Given this equivalent representation, an estimator of $\Sigma_{Y\ddot{X}|Z}$ can be constructed by plugging-in the empirical estimators of covariance operators in the r.h.s.~of \eqref{Eq:equiv}, yielding
\begin{equation}
\hat{\Sigma}_{Y\ddot{X}|Z}:=\hat{\Sigma}_{Y\ddot{X}}-\hat{\Sigma}_{YZ}(\hat{\Sigma}_{ZZ}+\lambda I)^{-1}\hat{\Sigma}_{Z\ddot{X}},\label{Eq:estim-cov}
\end{equation}
where $\lambda>0$ is the regularization parameter. Since $\hat{\Sigma}_{ZZ}$ is a finite rank infinite dimensional operator from $\H_\Z$ to $\H_\Z$, it is not invertible. Therefore an estimator of $\tilde{\Sigma}^{-1}_{ZZ}$ is obtained by regularizing $\hat{\Sigma}_{ZZ}$ by $\lambda I$ and inverting the result as shown in \eqref{Eq:estim-cov}.
For example, based on \eqref{Eq:cov-op}, an empirical estimator of $\Sigma_{YZ}$ can be written as 
$$\hat{\Sigma}_{YZ}=\frac{1}{n}\sum^n_{i=1}k_\Y(\cdot,Y_i)\otimes k_\Z(\cdot,Z_i)-\left(\frac{1}{n}\sum^n_{i=1}k_\Y(\cdot,Y_i)\right)\otimes \left(\frac{1}{n}\sum^n_{i=1}k_\Z(\cdot,Z_i)\right).$$ Similarly, the empirical estimators of $\hat{\Sigma}_{Y\ddot{X}}$, $\hat{\Sigma}_{ZZ}$ and $\hat{\Sigma}_{Z\ddot{X}}$ can be defined. An important point to note is that \eqref{Eq:estim-cov} does not require density estimation unlike in \eqref{Eq:v-hscic} but involves a nuisance parameter $\lambda$ as a counterpart to the bandwidth parameter in density estimation. The following result (Proposition~\ref{pro:estim}) provides a simple expression for $\Vert \hat{\Sigma}_{Y\ddot{X}|Z}\Vert^2_{HS}$ from which it is easy to see that $\Vert \Sigma_{Y\ddot{X}|Z}\Vert^2_{HS}$ can be estimated by an estimator that has a computational complexity of $O(n^3)$, thereby improving upon \eqref{Eq:v-hscic} and matching the computational complexity of $\mathbb{E}_Z[\Vert\dot{\Sigma}_{YX|Z}\Vert^2_{HS}]$.
\begin{proposition}\label{pro:estim}
Define $\bm{H}=\frac{1}{n}\left(I-\frac{1}{n}\bm{1}\bm{1}^\top\right)$, $\tilde{\bm{K}}_X=\bm{K}_X\bm{H}$, $\tilde{\bm{K}}_Y=\bm{K}_Y\bm{H}$, $\tilde{\bm{K}}_Z=\bm{K}_Z\bm{H}$, $\tilde{\bm{K}}_{\ddot{X}}=(\bm{K}_X\circ \bm{K}_Z)\bm{H}$ where $\circ$ denotes the Hadamard product, $[\bm{K}_X]_{ij}=k_\X(X_i,X_j)$, $[\bm{K}_Y]_{ij}=k_\Y(Y_i,Y_j)$ and $[\bm{K}_Z]_{ij}=k_\Z(Z_i,Z_j)$ for all $i,j=1,\ldots,n$. Then
\begin{equation}
\Vert\hat{\Sigma}_{Y\ddot{X}|Z}\Vert^2_{HS}=\emph{Tr}\left[\tilde{\bm{K}}_{\ddot{X}}\left(n \bm{H}-\left(\tilde{\bm{K}}_Z+\lambda I\right)^{-1}\tilde{\bm{K}}_Z\right)\tilde{\bm{K}}_Y\left(n \bm{H}-\left(\tilde{\bm{K}}_Z+\lambda I\right)^{-1}\tilde{\bm{K}}_Z\right)\right].\label{Eq:kenji}
\end{equation}
\end{proposition}
\begin{proof}
Define $$S_X:\H_\X\rightarrow\mathbb{R}^n,\,\,\,\, f\mapsto (f(X_1),\ldots,f(X_n))^\top.$$ It is easy to verify that $$S^*_X:\mathbb{R}^n\rightarrow \H_\X,\,\,\,\, \bm{\alpha}\mapsto \sum^n_{i=1}\alpha_i k_\X(\cdot,X_i),$$ where $\mathbf{\alpha}=(\alpha_1,\ldots,\alpha_n)^\top$ since for all $f\in\H_\X$ and $\bm{\alpha}\in\mathbb{R}^n$,
$$\langle S^*_X\bm{\alpha},f\rangle_2=\langle \bm{\alpha},S_X f\rangle_{\H_\X}=\sum^n_{i=1}\alpha_i f(X_i)=\sum^n_{i=1}\alpha_i\langle k_\X(\cdot,X_i),f\rangle_{\H_\X}=\left\langle \sum^n_{i=1}\alpha_i k_\X(\cdot,X_i),f\right\rangle_{\H_\X}.$$
Similarly define $$S_Y:\H_\Y\rightarrow\mathbb{R}^n,\,\,\,\, g\mapsto (g(Y_1),\ldots,g(Y_n))^\top$$ and $$S_Z:\H_\Z\rightarrow\mathbb{R}^n,\,\,\,\, h\mapsto (h(Y_1),\ldots,h(Y_n))^\top.$$ 
It is easy to verify that $$\hat{\Sigma}_{ZZ}=\frac{1}{n}S^*_Z\left(I-\frac{1}{n}\bm{1}\bm{1}^\top\right)S_Z=S^*_Z\bm{H}S_Z$$ since for any $h\in\H_\Z$,
$$\hat{\Sigma}_{ZZ}h=\frac{1}{n}\sum^n_{i=1}k_\Z(\cdot,Z_i)h(Z_i)-\left(\frac{1}{n}\sum^n_{i=1}k_\Z(\cdot,Z_i)\right)\left(\frac{1}{n}\sum^n_{i=1}h(Z_i)\right)=\frac{1}{n}S^*_ZS_Zh-\frac{1}{n^2}S^*_Z\bm{1}\bm{1}^\top S_Zh.$$
Similarly, it can be shown that 
$$\hat{\Sigma}_{YZ}=S^*_Y\bm{H}S_{Z},\,\,\,\hat{\Sigma}_{Y\ddot{X}}=S^*_Y\bm{H}S_{\ddot{X}}\,\,\,\text{and}\,\,\,\hat{\Sigma}_{Z\ddot{X}}=S^*_Z\bm{H}S_{\ddot{X}},$$
where $S_{\ddot{X}}:\H_\X\otimes \H_\Z\rightarrow\mathbb{R}^n$, $f\otimes h\mapsto (f(X_1)h(Z_1),\ldots,f(X_n)h(Z_n))^\top$. Therefore,
\begin{eqnarray}
\hat{\Sigma}_{Y\ddot{X}|Z}&{}={}&\hat{\Sigma}_{Y\ddot{X}}-\hat{\Sigma}_{YZ}(\hat{\Sigma}_{ZZ}+\lambda I)^{-1}\hat{\Sigma}_{Z\ddot{X}}\nonumber\\
&{}={}&S^*_Y\bm{H}S_{\ddot{X}}-S^*_Y\bm{H}S_Z\left(S^*_Z\bm{H}S_Z+\lambda I\right)^{-1}S^*_Z\bm{H}S_{\ddot{X}}\nonumber\\
&{}\stackrel{(*)}{=}{}&S^*_Y\bm{H}S_{\ddot{X}}-S^*_Y\bm{H}\left(S_ZS^*_Z\bm{H}+\lambda \bm{I}\right)^{-1}S_ZS^*_Z\bm{H}S_{\ddot{X}}\nonumber\\
&{}\stackrel{(**)}{=}{}&S^*_Y\bm{H}S_{\ddot{X}}-S^*_Y\bm{H}\left(\bm{K}_Z\bm{H}+\lambda \bm{I}\right)^{-1}\bm{K}_Z\bm{H}S_{\ddot{X}}\nonumber\\
&{}\stackrel{(\dagger)}{=}{}&S^*_Y\bm{H}\left(n \bm{I}-\left(\bm{K}_Z\bm{H}+\lambda \bm{I}\right)^{-1}\bm{K}_Z\right)\bm{H}S_{\ddot{X}},\nonumber
\end{eqnarray}
where in $(*)$ we used the identity $B(AB+\lambda I)^{-1}=(BA+\lambda I)^{-1}B$ which is easy to verify by noting that $(BA+\lambda I)B(AB+\lambda I)^{-1}(AB+\lambda I)=(BA+\lambda I)(BA+\lambda I)^{-1}B(AB+\lambda I)$. Since it is clear that for any $\bm{\alpha}\in\mathbb{R}^n$, $S_ZS^*_Z\bm{\alpha}=\sum^n_{i=1}\alpha_iS_Z k_\Z(\cdot,Z_i)=\bm{K}_Z\bm{\alpha}$, we used $\bm{K}_Z=S_ZS^*_Z$ in $(**)$. In $(\dagger)$, we used the fact that $\bm{H}^2=\bm{H}/n$. Therefore, it follows that
\begin{eqnarray}
\Vert \hat{\Sigma}_{Y\ddot{X}|Z}\Vert^2_{HS}&{}={}&\left\langle S^*_Y\bm{H}\left(n \bm{I}-\left(\bm{K}_Z\bm{H}+\lambda \bm{I}\right)^{-1}\bm{K}_Z\right)\bm{H}S_{\ddot{X}},S^*_Y\bm{H}\left(n \bm{I}-\left(\bm{K}_Z\bm{H}+\lambda \bm{I}\right)^{-1}\bm{K}_Z\right)\bm{H}S_{\ddot{X}}\right\rangle_{HS}\nonumber\\
&{}={}&\text{Tr}\left[S^*_{\ddot{X}}\bm{H}\left(n \bm{I}-\bm{K}_Z\left(\bm{H}\bm{K}_Z+\lambda \bm{I}\right)^{-1}\right)\bm{H}S_YS^*_Y\bm{H}\left(n \bm{I}-\left(\bm{K}_Z\bm{H}+\lambda \bm{I}\right)^{-1}\bm{K}_Z\right)\bm{H}S_{\ddot{X}}\right]\nonumber\\
&{}={}&\text{Tr}\left[\bm{K}_{\ddot{X}}\bm{H}\left(n \bm{I}-\bm{K}_Z\left(\bm{H}\bm{K}_Z+\lambda \bm{I}\right)^{-1}\right)\bm{H}\bm{K}_Y\bm{H}\left(n \bm{I}-\left(\bm{K}_Z\bm{H}+\lambda \bm{I}\right)^{-1}\bm{K}_Z\right)\bm{H}\right]\nonumber\\
&{}={}&\text{Tr}\left[\bm{K}_{\ddot{X}}\bm{H}\left(n \bm{H}-\left(\bm{K}_Z\bm{H}+\lambda \bm{I}\right)^{-1}\bm{K}_Z\bm{H}\right)\bm{K}_Y\bm{H}\left( n \bm{H}-\left(\bm{K}_Z\bm{H}+\lambda \bm{I}\right)^{-1}\bm{K}_Z\bm{H}\right)\right]\nonumber
\end{eqnarray}
and the result follows.
\end{proof}
Using arguments based on concentration of measure in separable Hilbert spaces, the consistency of the estimator in Proposition~\ref{pro:estim} along with the consistency of the conditional independence test when \eqref{Eq:kenji} is used as a test statistic has been established \citep{fukumizu2008kernel}. It is therefore clear from the above discussion that $\Vert\hat{\Sigma}_{Y\ddot{X}|Z}\Vert^2_{HS}$ has similar statistical and computational behavior to that $\mathbb{E}_Z[\Vert \dot{\Sigma}_{YX|Z}\Vert^2_{HS}]$, while not requiring density estimation.

\section{Discussion}
Conditional distance covariance is a commonly used metric for measuring conditional independence in the statistics community. In the machine learning community, a conditional independence measure based on reproducing kernels is popularly used in applications such as conditional independence testing. In this work, we have explored the connection between these two conditional independence measures where we showed the distance based measure to be a limiting version of the kernel based measure, where we may view conditional distance covariance as a member of a much larger class of kernel-based conditional independence measures. This may enable to design more powerful conditional independence tests by choosing a richer class of kernels. We also explored the empirical estimators of these two conditional independence measures where we showed that the kernel-based measure does not require density estimation unlike the distance-based measure, while both the empirical estimators have similar computational complexity.

\section*{Acknowledgements}
BKS is partially supported by NSF-DMS-1713011.
\bibliographystyle{apalike}
\bibliography{references}

\end{document}